\documentclass[12pt, reqno]{amsart}
\usepackage{amscd,amsmath,amsthm,amssymb,graphics}
\usepackage{amsfonts,amssymb,amscd,amsmath,enumitem,verbatim}
\usepackage[a4paper,top=3cm,left=3cm,right=3cm]{geometry}
\usepackage{xcolor}
\usepackage{float}
\theoremstyle{plain}
\usepackage{booktabs}
\usepackage{hyperref}
\newtheorem{Theorem}{Theorem}

\newtheorem{Proposition}[Theorem]{Proposition}

\newtheorem{Conjecture}[Theorem]{Conjecture}

\newtheorem{Definition}[Theorem]{Definition}
\newtheorem{Remark}[Theorem]{Remark}

\newtheorem{Example}[Theorem]{Example}
\newtheorem{Assumption}[Theorem]{Assumption}

\title{Real convergence and periodicity of $p$--adic continued fractions}
\author{Giuliano Romeo}

\begin{document}

\maketitle

\begin{center}
Department of Mathematical Sciences Giuseppe Luigi Lagrange,\\
Politecnico di Torino, Corso Duca degli Abruzzi 24, Torino
\ \\ \ \ \\
giuliano.romeo@polito.it
\end{center}

\begin{abstract}
Continued fractions have been generalized over the field of $p$--adic numbers, where it is still not known an analogue of the famous Lagrange's Theorem. In general, the periodicity of $p$--adic continued fractions is well studied and addressed as a hard problem. In this paper, we show a strong connection between periodic $p$--adic continued fractions and the convergence to real quadratic irrationals. In particular, in the first part we prove that the convergence in $\mathbb{R}$ is a necessary condition for the periodicity of the continued fraction of a quadratic irrational in $\mathbb{Q}_p$. Moreover, we leave several conjectures on the converse, supported by experimental computations. In the second part of the paper, we exploit these results to develop a probabilistic argument for the non-periodicity of Browkin's $p$-adic continued fractions. The probabilistic results are conditioned under the assumption of uniform distribution of the $p$--adic digits of a quadratic irrational, that holds for almost all $p$--adic numbers.
\end{abstract}

\section{Introduction}
Continued fractions are expressions of the form
\begin{equation}
a_0+\cfrac{1}{a_1+\cfrac{1}{a_2+\ddots}},
\end{equation}
and they are a very classical and powerful tool in number theory. They have been studied and used throughout the centuries in all areas of mathematics, mainly due to their excellent properties of approximation. Continued fractions have been employed, for example, to prove the irrationality of $\pi$ \cite{LAM}, to construct transcendental numbers \cite{AB2,LIO}, to find indecomposable algebraic integers in quadratic number fields \cite{DS} and for  the factorization of integers \cite{LP}. To have a fairly complete introduction on the classical theory of continued fractions, see \cite{KHI, OLDS, WALL}. In the field of real numbers, continued fractions completely characterize rational numbers and quadratic irrationals. A real number $\alpha\in\mathbb{R}$ is rational if and only if its continued fraction is finite. The latter is an easy consequence of the finiteness of the Euclidean division algorithm in $\mathbb{Z}$. Moreover, $\alpha\in\mathbb{R}$ is a quadratic irrational, i.e. a root of an irreducible rational polynomial of degree $2$, if and only if its continued fraction is eventually periodic. The latter is the famous Lagrange's Theorem \cite{LAG}. Because of their optimal properties, continued fractions are highly studied in mathematics and there exist several generalizations. The topic of this paper is the study of continued fractions in the field of $p$--adic numbers $\mathbb{Q}_p$, introduced by Mahler in \cite{MAH}. See \cite{R} for a recent survey on the theory of $p$--adic continued fractions. The problem of defining a continued fraction algorithm in $\mathbb{Q}_p$ sharing all the optimal properties enjoyed by classical continued fractions is still open. In particular, it is not known an algorithm that provides an eventually periodic continued fraction for every quadratic irrational in $\mathbb{Q}_p$. Hence, $p$--adic continued fractions miss, up to now, an analogue of Lagrange's Theorem for classical continued fractions.\smallskip

The first definitions date back around the 1970s by Ruban \cite{RUB}, Schneider \cite{SCH} and Browkin \cite{BI}. The reason of the non-existence of a unique standard algorithm is that there is not a unique choice for the floor function of a $p$--adic number. The choices of Ruban, Schneider and Browkin are all meaningful with respect to different points of view. However, Ruban's and Schneider's continued fractions are not finite for every rational number, but they can also be infinite periodic  \cite{BUN,LAO}. Moreover, these two algorithms do not provide a periodic continued fraction for every $p$--adic quadratic irrational. In fact, the analogues of Lagrange's Theorem have been disproved in \cite{TIL} for Schneider's algorithm and in \cite{CVZ} for Ruban's algorithm. In 1978 and 2000, Browkin \cite{BI,BII} came out with the definition of two $p$--adic continued fractions algorithms (we call them \textit{Browkin I} and \textit{Browkin II}) that terminate if and only if the input is a rational number. The proof of finiteness for \textit{Browkin I} is contained in Browkin's first paper \cite{BI}, while for \textit{Browkin II} was conjectured by Browkin in \cite{BII} and proved later on by Barbero, Cerruti and Murru in \cite{BCMI}. Browkin's first algorithm is really similar to Ruban's one, but it allows to have also negative numbers as partial quotients. In fact, Browkin's choice is to take, in the $p$--adic expansion of an element of $\mathbb{Q}_p$, the representatives modulo $p$ in $\{-\frac{p-1}{2},\ldots,\frac{p-1}{2}\}$ instead of $\{0,\ldots,p-1\}$. We will refer to these kind of algorithms, choosing representatives in the symmetric interval $\{-\frac{p-1}{2},\ldots,\frac{p-1}{2}\}$, as \textit{Browkin-type} algorithms.\smallskip

The problem of understanding if \textit{Browkin-type} continued fractions are always periodic for $p$--adic quadratic irrationals or not is still open, and it is usually addressed as a hard problem. The main difference with Ruban's and Schneider's cases is the presence of negative partial quotients in the continued fraction expansion (see Section \ref{Sec: generalremarks} for more details). The periodicity of \textit{Browkin I} has been studied by Bedocchi in \cite{BEI,BEII} and by Capuano, Murru and Terracini in \cite{CMT}. The periodicity of \textit{Browkin II} has been deepened in \cite{BCMI,MRSI}. In particular, \textit{Browkin II} computationally appears to outperform \textit{Browkin I} in the number of periodic continued fractions provided for quadratic irrational numbers. Moreover, the convergence conditions proved in \cite{MRSII} allowed to define a new algorithm in \cite{MR}, that is introduced in Section \ref{Sec: preliminari} as Algorithm \eqref{Alg: MR}. The algorithm defined in \cite{MR} is similar to \textit{Browkin II} but it improves its periodicity properties from both theoretical and computational points of view. Other algorithms have been defined and studied in \cite{BCMII, DW} for standard $p$--adic continued fractions and in \cite{MT1,MT2,MT3} for the multidimensional case. See also \cite{MULA,CMT2} for other generalizations of $p$--adic continued fractions. All these algorithms improve and generalize the known algorithms for $p$--adic continued fractions from several different points of view.\bigskip

The main object of study of the present paper is the periodicity of the \textit{Browkin-type} algorithms defined in \cite{BI,BII,MR} and introduced in the next section, i.e. \textit{Browkin I}, \textit{Browkin II} and Algorithm \eqref{Alg: MR}. Although Lagrange's Theorem has not been proved or disproved for any these three algorithms, it is widely believed to fail (see, for example, the computational analysis performed in \cite{MR}). However, as we discuss in Section \ref{Sec: generalremarks}, it does not exist at the present time an argument to deal efficiently with the periodicity of \textit{Browkin-type} continued fractions. The main reason is that the argument carried out in the pioneering work of Capuano, Veneziano and Zannier \cite{CVZ} for the periodicity of Ruban's continued fractions, is specific for continued fractions with positive partial quotients. For these reasons, at the present time, no quadratic irrational has been proved to have non-periodic $p$--adic continued fraction with a \textit{Browkin-type} algorithm.\bigskip

In this paper, we explore a different and novel argument for the non-periodicity of \textit{Browkin-type} continued fractions. In particular, we highlight a strong connection between the periodicity of a $p$--adic continued fraction and the convergence in $\mathbb{R}$ to a real quadratic irrational. The core result and the starting point for our analysis is the following necessary condition for the periodicity.
\begin{Proposition}\label{Pro: Conv}
Let $\alpha$ be a quadratic irrational that has both a real image $\alpha^{(r)}$ and a $p$--adic image $\alpha^{(p)}$. If the $p$--adic continued fraction $[a_0,a_1,\ldots]$ of $\alpha^{(p)}$ obtained with \textit{Browkin I}, \textit{Browkin II} or Algorithm \eqref{Alg: MR} is eventually periodic, then $[a_0,a_1,\ldots]$ converges in $\mathbb{R}$ to either $\alpha^{(r)}$ or $\overline{\alpha}^{(r)}$ .
\end{Proposition} 

Proposition \ref{Pro: Conv} can be proved by using a result for the convergence of generalized continued fractions with complex partial quotients \cite[Theorem 8.1]{WALL}, based on the results of Lane \cite{LAN}, and it is discussed more thoroughly in Section \ref{Sec: realconve}. The idea is that, by requiring the periodicity, the continued fraction is a root of a polynomial of degree $2$ over $\mathbb{Q}$, and this is true both embedding the continued fraction in $\mathbb{Q}_p$ and in $\mathbb{R}$. Moreover, Proposition \ref{Pro: Conv} was also conjectured in a short note by Moore, that is not anymore available online.\bigskip

In Section \ref{Sec: generalremarks}, we highlight as working by embedding $p$--adic continued fractions into the field of real numbers has been, up to now, one of the most useful techniques to obtain effective results. In fact, it is the core idea of the nice characterization of the periodicity of Ruban's continued fractions proved in \cite{CVZ}. Also the strategy for the proof of finiteness of \textit{Browkin I} for all rational numbers, carried out in \cite{BI}, exploits an inequality for the Euclidean absolute value of the partial quotients. Moreover, it is also the underlying idea of the algorithms defined in \cite{DW}, producing periodic $p$--adic continued fractions of quadratic irrationals for small values of $p$.\bigskip

From the experimental computations, it seems that the convergence in $\mathbb{R}$ to either $\alpha^{(r)}$ or $\overline{\alpha}^{(r)}$ of the $p$--adic continued fraction of $\alpha^{(p)}$ is also a sufficient condition for periodicity, i.e. also the converse of Proposition \ref{Pro: Conv} is true. In fact, in all the cases where the continued fraction $[a_0,a_1,\ldots]$ of a quadratic irrational $\alpha^{(p)}$ seems to be non-periodic, as a period is not detected for many steps, then $[a_0,a_1,\ldots]$ seems converging in $\mathbb{R}$ very neatly to a limit different from $\alpha^{(r)}$ (see the computational analysis performed in Section \ref{Sec: computations}). For this reason, we leave the following conjecture.

\begin{Conjecture}\label{Conj: realconv}
Let $[a_0,a_1,\ldots]$ be the \textit{Browkin I}, \textit{Browkin II} or Algorithm \eqref{Alg: MR} continued fraction of a $p$-adic quadratic irrational $\alpha^{(p)}$ that can be embedded in the real numbers. Let us call $\alpha^{(r)}$ the image of $\alpha^{(p)}$ in $\mathbb{R}$. Then, the continued fraction $[a_0,a_1,\ldots]$ is eventually periodic if and only if it converges with respect to the Euclidean absolute value to either $\alpha^{(r)}$ or $\overline{\alpha}^{(r)}$.
\end{Conjecture}

In the next sections, we provide several examples in support of Conjecture \ref{Conj: realconv}. After many steps, the continued fraction of a quadratic irrational $\alpha^{(p)}$ tends to oscillate less and less around a real limit, and the latter seems to be, in the cases where no periodicity is observed, really far from $\alpha^{(r)}$ and $\overline{\alpha}^{(r)}$. Moreover, a big deviation between consecutive convergents must correspond to small values of the partial quotients $a_n$. Small values of the partial quotients are very unlikely, especially for large values of the prime $p$. In fact, it is widely believed that the digits of the $p$--adic expansions of irrational numbers are uniformly distributed in $\mathbb{Z}/p\mathbb{Z}$. Therefore, in Section \ref{Sec: prob}, we develop a probabilistic argument for the non-periodicity of $p$--adic continued fractions obtained with \textit{Browkin I}, and similar arguments can be applied for \textit{Browkin II} and Algorithm \eqref{Alg: MR}. The probabilistic results are conditioned under the following assumption.

\begin{Assumption}\label{Ass: equi}
Let us consider $\alpha=\sum\limits_{n=r}^{+\infty}c_n p^n$ be a random $p$--adic quadratic irrational. Then, for all $n\geq r$ and for all $k\in\{0,\ldots,p-1\}$,
\[\mathbb{P}(c_n=k)=\frac{1}{p},\]
with respect to the Haar measure. In other words, the coefficients of the $p$-adic expansion of $\alpha$ are uniformly distributed in the set $\{0,\ldots,p-1\}$.
\end{Assumption}

The reasonability of Assumption \ref{Ass: equi} is based on Borel's normal number theorem \cite{BOR}. In 1909, Borel proved that \textit{almost all} real numbers are normal, i.e. the set of non-normal numbers has Lebesgue measure $0$ in $\mathbb{R}$. A real number $\alpha$ is (absolutely) normal if, for any integer $b\geq 2$, every digit of its expansion in base $b$ has density $\frac{1}{b}$. It is known that Borel's argument, exploiting Borel-Cantelli Lemma, can be applied equivalently to the $p$--adic expansion of $p$--adic numbers. Therefore, almost all $p$--adic numbers are normal. In particular, all irrational numbers, both in $\mathbb{R}$ and in $\mathbb{Q}_p$, are conjectured to be normal. However, apart from Borel's existence result, only few numbers have been proved to be normal. For example, there is not a proof that $\sqrt{2}$, $\pi$ or $e$ are normal, although it is strongly considered true.\bigskip

In Theorem \ref{Thm: expvalue} of Section \ref{Sec: prob}, under Assumption \ref{Ass: equi}, we prove that the expected value of any \textit{Browkin I} partial quotient $a$ is
\begin{equation}\label{Eq: Ea}
\mathbb{E}(|a|)=\frac{p}{4}\left(1-\frac{1}{p^2(p^2+p+1)}\right),
\end{equation}
therefore it grows linearly with the prime $p$. This means that the convergents are really unlikely to deviate after several steps, especially for large $p$. For instance, see Example \ref{Exa: deviation} at the end of Section \ref{Sec: prob}, where the $7823$--adic continued fraction of $\sqrt{15648}$ has the convergent
\begin{equation}\label{Eq: A10000}
\frac{A_{10000}}{B_{10000}}\approx 3337.
\end{equation}

If its \textit{Browkin I} expansion became periodic, then, by Proposition \ref{Pro: Conv}, it should be convergent to $\sqrt{15648}\approx 125$. However, it is really unlikely for the subsequent convergents to deviate from \eqref{Eq: A10000}, as $\mathbb{E}(|a_n|)\approx 1955.75$.\bigskip

Notice that the growth of the denominators $|B_n|$ tends to be exponential, as the main term is $|a_0a_1\ldots a_n|$. Therefore, we leave the following conjecture.

\begin{Conjecture}
Let $\{|B_n|\}_{n\in\mathbb{N}}$ be the sequence of denominators of convergents for \textit{Browkin I}, \textit{Browkin II} and Algorithm \eqref{Alg: MR} $p$-adic continued fractions, where $p\geq 5$. Then $|B_n|$ tends to $+\infty$ with probability $1$.
\end{Conjecture}

\section{Preliminaries and main algorithms}\label{Sec: preliminari}
In this section we recall some basic facts from the theory of continued fractions and $p$--adic numbers. For more results and background we refer the reader to \cite{OLDS, WALL} for continued fractions and to \cite{FG} for $p$--adic numbers.\bigskip

In the following, let $p$ be an odd prime number. We denote by $v_p(\cdot)$ the $p$--adic valuation and by $|\cdot|$ and $|\cdot|_p$, respectively, the Euclidean absolute value and the $p$--adic absolute value. We denote the continued fraction expansion 
\begin{equation}
a_0+\cfrac{1}{a_1+\ddots}
\end{equation}
by $[a_0,a_1,\ldots]$. The coefficients $a_n$ are called \textit{partial quotients} and they are, for our purposes, rational numbers. For all $n\in\mathbb{N}$, the rational number
\begin{equation}\label{Eq: convergents}
\frac{A_n}{B_n}=[a_0,a_1,\ldots,a_n]=
a_0+\cfrac{1}{a_1+\ddots + \cfrac{1}{a_{n-1}+\cfrac{1}{a_n}}},
\end{equation}

corresponding to the continued fraction stopped at the $n$-th term, is called $n$-th \textit{convergent} of the continued fraction. It is not hard to prove that the sequences $\{A_n\}_{n\in\mathbb{N}}$ and $\{B_n\}_{n\in\mathbb{N}}$ in \eqref{Eq: convergents} satisfy the following recursions:
\begin{equation}\label{Eq: Recursions}
\begin{cases}
A_0=a_0,\\
A_1=a_1a_0+1,\\
A_n=a_nA_{n-1}+A_{n-2}, \ \ n \geq 2,
\end{cases}
\begin{cases}
B_0=1,\\
B_1=a_1,\\
B_n=a_nB_{n-1}+B_{n-2}, \ \ n \geq 2.
\end{cases}    
\end{equation}
Moreover, they satisfy, for all $n\in\mathbb{N}$, the relation:
\begin{equation}\label{Eq: converelation}
A_nB_{n+1}-A_{n+1}B_{n}=(-1)^{n+1}
\end{equation}

The standard algorithm to express a real number $\alpha$ through a simple continued fraction $[a_0,a_1,\ldots]$ works as follows, starting with $\alpha_0=\alpha$:
\begin{equation}\label{Alg: R}
\begin{cases}
a_n=\lfloor \alpha_n \rfloor \\
\alpha_{n+1}=\frac{1}{\alpha_n-a_n},
\end{cases}
\end{equation}

where $\lfloor \alpha_n \rfloor $ denotes the \textit{integer part} of $\alpha$. The elements $\alpha_n$ are called \textit{complete quotients}. If $\alpha_n=a_n$ for some $n\in\mathbb{N}$, then the algorithm terminates and the continued fraction is finite.\bigskip

One of the main strategies for introducing continued fractions in $\mathbb{Q}_p$ has been to search for an analogue of the floor function $\lfloor \cdot \rfloor$ that is meaningful inside $\mathbb{Q}_p$, in order to emulate the standard algorithm \eqref{Alg: R}. The first algorithms have been proposed by Ruban \cite{RUB} and Schneider \cite{SCH}. Browkin's first algorithm \cite{BI} is very similar to Ruban's algorithm, with the exception that the representatives of $\mathbb{Z}/p\mathbb{Z}$ are chosen in $\{-\frac{p-1}{2},\ldots,\frac{p-1}{2}\}$ instead of $\{0,\ldots,p-1\}$. This small variation is fundamental because Browkin's algorithm produces a finite continued fraction for each rational number, while Ruban's algorithm can also be periodic over the rationals. Given $\alpha=\sum\limits_{i=-r}^{+\infty}c_ip^i\in\mathbb{Q}_p$, with $c_i \in \{ -\frac{p-1}{2}, \ldots, \frac{p-1}{2} \}$, Browkin defines the floor function $s:\mathbb{Q}_p\rightarrow \mathbb{Q}$ as
\begin{equation}\label{Eq: sfunc}
s(\alpha)=\sum\limits_{i=-r}^{0}c_ip^i,
\end{equation}
and $s(\alpha)=0$ if $r<0$. Given the floor function $s$, the algorithm is the same as \eqref{Alg: R}. Browkin's first algorithm, that we call \textit{Browkin I}, work as follows. At the first step $\alpha_{0}=\alpha$ and, for all $n\geq 0$,
\begin{align} 
\begin{cases}\label{Alg: Br1}
a_n=s(\alpha_n)\\
\alpha_{n+1}=\frac{1}{\alpha_n-a_n}.
\end{cases}
\end{align}
If at some point $\alpha_n=a_n$, then the algorithm stops and $\alpha=[a_0,\ldots,a_n]$, i.e. $\alpha$ has finite \textit{Browkin I} continued fraction.\smallskip

In \cite{BII}, Browkin defined another floor function that is similar to the first function $s$, but excluding the summand with valuation zero. For $\alpha=\sum\limits_{i=-r}^{+\infty}c_ip^i\in\mathbb{Q}_p$, with $c_i \in \{ -\frac{p-1}{2}, \ldots, \frac{p-1}{2} \}$, the second floor function is the function $t:\mathbb{Q}_p\rightarrow \mathbb{Q}$, such that
\[t(\alpha)=\sum\limits_{i=-r}^{-1}c_ip^i,\]
and $t(\alpha)=0$ if $r\leq 0$. The second algorithm, that we call \textit{Browkin II}, works on an input $\alpha$ as follows. At the first step $\alpha_{0}=\alpha$ and, for all $n\geq 0$,
\begin{align} 
\begin{cases}\label{Alg: Br2}
a_n=s(\alpha_n) \ \ \ \ \ & \textup{if} \ n \ \textup{even}\\
a_n=t(\alpha_n) & \textup{if} \ n \ \textup{odd}\ \textup{and} \ v_p(\alpha_n-t(\alpha_n))= 0\\
a_n=t(\alpha_n)-sign(t(\alpha_n)) & \textup{if} \ n \ \textup{odd} \ \textup{and} \ v_p(\alpha_n-t(\alpha_n))\neq 0\\
\alpha_{n+1}=\frac{1}{\alpha_n-a_n}.
\end{cases}
\end{align}
If at some point $\alpha_n=a_n$, then the algorithm stops and $\alpha=[a_0,\ldots,a_n]$, i.e. $\alpha$ has finite \textit{Browkin II} continued fraction. One consequence of the alternation of the two functions is that, in \textit{Browkin II}, all the even partial quotients are by constructions integers and all the odd partial quotients are non-integer rationals. In \cite{MR}, a modification of \textit{Browkin II} has been defined, that does not use the sign function. For all $\alpha_0\in\mathbb{Q}_p$, the algorithm works as follows:
\begin{align} 
\begin{cases}\label{Alg: MR}
a_n=s(\alpha_n) \ \ \ \ \ & \textup{if} \ n \ \textup{even}\\
a_n=t(\alpha_n) & \textup{if} \ n \ \textup{odd}\\
\alpha_{n+1}=\frac{1}{\alpha_n-a_n}.
\end{cases}
\end{align}
It has been proved in \cite{MR} that Algorithm \eqref{Alg: MR} improves the periodicity properties of Browkin's second algorithm. For more details and a computational analysis on the performance of \textit{Browkin I}, \textit{Browkin II} and Algorithm \eqref{Alg: MR}, see \cite{MR}.

\section{General remarks on the periodicity}\label{Sec: generalremarks}
Lagrange's Theorem states that the standard continued fraction expansion of a real quadratic irrational number becomes eventually periodic. The most natural technique is to divide the proof in two steps. First, a theorem of Galois \cite{GAL} characterizes the purely periodic continued fractions.

\begin{Theorem}[Galois' Theorem]\label{Thm: Galois}
The continued fraction expansion of a quadratic irrational $\alpha\in\mathbb{R}$ is purely periodic if and only if it is \textit{reduced}, that is $\alpha>1$ and $-1<\overline{\alpha}<0$, where $\overline{\alpha}$ is the conjugate of $\alpha$ over $\mathbb{Q}$.
\end{Theorem}

The idea to prove Galois' Theorem is that if a complete quotient 
\begin{equation}\label{Eq: quadirra}
\alpha_n=\frac{P_n+\sqrt{D}}{Q_n}
\end{equation}
is \textit{reduced}, also $\alpha_{n+1}$ is \textit{reduced}. Then, it is not hard to see that requiring $|\alpha_n|>1$ and $-1<\overline{\alpha}_n<0$ gives rise to only finitely many different choices for $P_n$ and $Q_n$. Therefore, a repetition $\alpha_{n}=\alpha_{n+k}$, $k\geq 1$, is met at some point, and the continued fraction is periodic. The standard proof of Lagrange's Theorem consists in proving that the expansion of every quadratic irrational eventually meets a \textit{reduced} quadratic irrational among its complete quotients. Hence, it starts to be periodic at some point.\smallskip

In $\mathbb{Q}_p$ there are some characterizations of purely periodic continued fractions, that are really similar to Galois' Theorem both in the statement and in the proof. 
\begin{Theorem}[\cite{CVZ},\cite{BEI}]\label{Thm: pureBR}
Let $\alpha\in\mathbb{Q}_p$ having a periodic Ruban's (or Browkin's) continued fraction. Then the continued fraction is purely periodic if and only if $|\alpha|_p>1$ and $|\overline{\alpha}|_p<1$.
\end{Theorem}
\begin{Theorem}[\cite{MR}]\label{Thm: pureMR}
Let $\alpha\in\mathbb{Q}_p$ having a periodic continued fraction with Algorithm \eqref{Alg: MR}. Then the continued fraction is purely periodic if and only if $|\alpha|_p\geq1$ and $|\overline{\alpha}|_p<1$.
\end{Theorem}

In Theorems \ref{Thm: pureBR} and \ref{Thm: pureMR}, the hypothesis of starting from a periodic continued fraction can not be removed. This is in great contrast with Galois' Theorem, where the conditions on the absolute values of $\alpha$ and $\overline{\alpha}$ are also sufficient for the periodicity. \smallskip

The main problem here is the use of the $p$--adic absolute value instead of the Euclidean one. In fact, unlike the real case, there are infinitely many $p$--adic quadratic irrationals $\alpha$ satisfying $|\alpha|_p>1$ (or $\geq 1$) and $|\overline{\alpha}|_p<1$. Therefore, being \textit{reduced} is not a sufficient condition for periodicity in $\mathbb{Q}_p$.\smallskip

In \cite{CVZ}, Capuano, Veneziano and Zannier proved that there are only finitely many quadratic irrational with same discriminant $D$ having a periodic Ruban's continued fraction. This allowed to prove an effective criterion for determining in a finite number of steps if a given quadratic irrational becomes periodic or not.\smallskip

The idea carried out in \cite{CVZ} exploits a ``real" argument more than a ``$p$--adic" argument. In fact, it is used the real image of the quadratic irrational $\alpha$ in $\mathbb{R}$. If the $p$--adic quadratic irrational number
\[\alpha=\frac{P+\sqrt{D}}{Q}\]
has got a purely periodic expansion with Ruban's algorithm, then $\alpha=\alpha_k$ for some $k$, hence
\[\alpha=\frac{\alpha_kA_{k-1}+A_{k-2}}{\alpha_kB_{k-1}+B_{k-2}}=\frac{\alpha A_{k-1}+A_{k-2}}{\alpha B_{k-1}+B_{k-2}},\]
and it is a root of the polynomial
\[B_{k-1}\alpha^2-(A_{k-1}-B_{k-2})\alpha-A_{k-2}=0.\]
Since all the partial quotients in Ruban's expansion are positive, then the numerators $A_n$ and the denominators $B_n$ are positive for all $n\in\mathbb{N}$. Therefore, if $\alpha$ has a periodic expansion, then
\begin{equation}\label{Eq: CVZ}
\alpha\overline{\alpha}=-\frac{A_{k-2}}{B_{k-1}}<0.
\end{equation}
If $\alpha$ is purely periodic, then all its complete quotients $\alpha_n$ are purely periodic and hence satisfy \eqref{Eq: CVZ}. However, this means that
\begin{equation}\label{Eq: CVZ2}
\alpha_n\overline{\alpha}_n=\frac{P_n^2-D}{Q_n^2}<0,
\end{equation}
for all $n\in\mathbb{N}$. It is not hard to see that \eqref{Eq: CVZ2} gives an upper bound on both $P_n$ and $Q_n$, depending only on $D$, similar to those obtained in the real case. Therefore, there are only finitely many different complete quotients $\alpha_n$ for a $p$--adic number $\alpha$ having a purely periodic Ruban's expansion. This gives an effective practical criterion for determining the periodicity or non-periodicity of Ruban's continued fractions, since in finite time (practically determinable) there is either a repetition or a complete quotient not satisfying \eqref{Eq: CVZ}.\bigskip

Unfortunately, a similar technique is not applicable for \textit{Browkin-type} algorithms. In fact, since Browkin's partial quotients can also be negative, then \eqref{Eq: CVZ} is not a necessary condition for the periodicity. This is one of the main reasons why the problem of determining if an analogue of Lagrange's Theorem holds or fails is still unsolved. The best result in this sense is the following, and it has been proved by Capuano, Murru and Terracini in \cite{CMT} for \textit{Browkin I}.

\begin{Proposition}[\cite{CMT}, Proposition 4.8]\label{Prop: cmt}
Let $\alpha\in\mathbb{Q}_p$ be a quadratic irrational and for every $n>0$ denote by $\xi_n$ and $\xi'_n$ the two images of $\alpha_n$ in $\mathbb{C}$ and $t=\lfloor \sqrt{\Delta}\rfloor$. Assume that $\exists n_0>0$ such that $\xi_n\xi'_n<0$ for every $n\in\left[n_0,n_0+K\right]$, where $K=K(t)$ is a constant depending only on $t$. Then the \textit{Browkin I} continued fraction of $\alpha$ is periodic with a period of length at most $K$.  
\end{Proposition}

The proof of Proposition \ref{Prop: cmt} exploits an argument really similar to the one carried out in \cite{CVZ}, starting from Equation \eqref{Eq: CVZ}. 

\begin{Remark}
In order to apply the argument of \cite{CVZ} to obtain a criterion for the periodicity of \textit{Browkin-type} $p$--adic continued fractions, we do not necessarily need \eqref{Eq: CVZ}. It would suffice to have, for some $K>0$,
\begin{equation}\label{Eq: utopia}
\frac{A_{k-1}}{B_{k}}>-K,
\end{equation}
so that
\begin{equation}
\alpha\overline{\alpha}=-\frac{A_{k-1}}{B_{k}}<K.
\end{equation}
This would lead to some bound on the sequences $P_n$ and $Q_n$ and hence to an effective criterion for the periodicity of Browkin's continued fractions. However, even if the constant $K$ can be arbitrarily large, it is not easy (or probably not possible at all) to find a $p$-adic quadratic irrational $\alpha$ such that its convergents satisfy \eqref{Eq: utopia}.
\end{Remark}

For all these reasons, there is not an ``easy" way to prove non-periodicity for \textit{Browkin I} up to now, and similar reasonings can be performed for \textit{Browkin II} or Algorithm \eqref{Alg: MR}. At the present time, no quadratic irrational has been proved to have non-periodic \textit{Browkin I} continued fraction, although it is largely believed to fail Lagrange's Theorem (see, for example, the experimental computations in \cite{MR}).\bigskip

Before delving into the real convergence of $p$--adic continued fractions, let us express few other words on the effectiveness of the real approach for $p$--adic continued fractions. Browkin's proof \cite{BI} that every rational numbers has a finite \textit{Browkin I} continued fraction could not be carried out without an argument controlling the Euclidean size of the partial quotients inside $\mathbb{R}$. In fact, the standard proof of finiteness for rational numbers (that also inspired all the subsequent proofs of finiteness in \cite{BCMI,BCMII,DW,MR,MRSII}) consists in producing a strictly decreasing sequence of integers in Euclidean absolute value. It is not possible to use a similar argument inside $\mathbb{Q}_p$, since there are infinite sequences of integers that are strictly decreasing in $p$--adic absolute value. Hence, all the proofs of finiteness for the $p$--adic continued fractions of rational numbers that are known so far, are performed by embedding the continued fractions inside $\mathbb{R}$. It is remarkable also the technique carried out in \cite{MVV} for the archimedean setting and then in \cite{CMT2} for the $p$-adic one, in order to prove finiteness for continued fractions defined over quadratic number fields. The arguments use Weil height and Northcott's theorem. This is an interesting approach because Weil height takes into account all absolute values of an algebraic number, not just the Euclidean or the $p$-adic one for a fixed $p$.

In \cite{DW}, Deng and Wang defined a continued fraction algorithm following an idea similar to Browkin's proof of finiteness, but for proving the periodicity. They introduced two novel floor functions $\tilde{s}$ and $\tilde{t}$ that minimize the Euclidean absolute value of $\alpha-\tilde{s}(\alpha)$ and $\alpha-\tilde{t}(\alpha)$. This machinery allowed to find some nice inequalities for the numerators $P_n$ and the denominators $Q_n$ of the complete quotients $\alpha_n$. These inequalities led, for small values of the prime $p$, to a strictly decreasing sequence involving $P_n$ and $Q_n$. Therefore, for these small values of $p$, the $p$--adic continued fraction of every quadratic irrational is eventually periodic, using their algorithm. This is the first proof of periodicity of an algorithm in some $\mathbb{Q}_p$ for every $p$--adic quadratic irrational (at least for those having positive discriminant, hence a real embedding). A periodic continued fraction expression for every quadratic irrational $\alpha$ having both a $p$--adic and a real interpretation was already defined in \cite{BCMI}. In \cite{BCMI}, the authors exploited the simultaneous convergence to $\alpha$ in $\mathbb{R}$ and $\mathbb{Q}_p$ of the Rédei rational functions $R_n$ \cite{RED} and provided a periodic continued fraction having exactly the $R_n$ as convergents.\smallskip

However, the important weakness of these algorithms is that they work only for quadratic irrationals, and in particular for those that can be embedded in both $\mathbb{Q}_p$ and $\mathbb{R}$. Moreover, they require the explicit knowledge of the minimal polynomial of the quadratic irrational. This is in great contrast with the spirit of Lagrange's Theorem for classical continued fractions. In fact, the standard algorithm works on a generic $\alpha\in\mathbb{R}$ and it is eventually periodic if and only if $\alpha$ is a quadratic irrational.

\section{Real convergence of $p$--adic continued fractions}\label{Sec: realconve}

The problem in full generality could be stated as: when does a continued fraction $[a_0,a_1,\ldots]$ converge both in the $p$--adic and in the real topology? Not always, of course. For example, the well known continued fraction
\begin{equation}\label{Eq: phi}
[\overline{1}]=\frac{1+\sqrt{5}}{2},
\end{equation}

converges in $\mathbb{R}$ to the golden mean, but does not converge in $\mathbb{Q}_p$. Moreover, also a continued fraction $[a_0,a_1,\ldots]$ with rational partial quotients that converges in $\mathbb{Q}_p$ not necessarily converges in $\mathbb{R}$ (not even in $\mathbb{C}$). For example
\[\left[\frac{1}{p},\frac{1}{p^2},\frac{1}{p^3},\ldots\right]\]
converges to a $p$--adic number since $v_p(a_n)<0$ for all $n$, but the series 
\[\sum\limits_{n=0}^{\infty}|a_n|=\sum\limits_{n=0}^{\infty}\frac{1}{p^{n+1}}=\frac{1}{p-1},\]
converges in Euclidean absolute value. Therefore, the continued fraction
$[a_0,a_1,\ldots]$ diverges to $+\infty$ in the Euclidean setting, due to a known result that we recall here.

\begin{Proposition}[\cite{WALL}, Theorem 6.1]
If the real continued fraction $[a_0,a_1,\ldots]$ converges to a real number, then the series $\sum\limits_{n=0}^{+\infty}|a_n|$ diverges to $+\infty$.
\end{Proposition}

Convergence is usually easier for periodic continued fractions and there are much more precise results than the non-periodic case. However, as we have seen in \eqref{Eq: phi}, periodic continued fractions that converge in $\mathbb{R}$ not necessarily converge in $\mathbb{Q}_p$. Also periodic $p$--adic continued fractions can be not convergent with respect to the Euclidean absolute value. For example, if we consider $[\overline{\frac{2}{5},-\frac{2}{5}}]$, then
\begin{equation}\label{Eq: negdis}
\alpha=\frac{\alpha_2A_1+A_0}{\alpha_2B_1+B_0}=\frac{\frac{21}{25}\alpha+\frac{2}{5}}{-\frac{2}{5}\alpha+1}.
\end{equation}

This means that $\alpha$ is a root of $5x^2-2x+5$, that has negative discriminant $\Delta=-96$. Therefore, it should converge to a non-real complex number, which is clearly not possible, as the convergents are rationals and $\mathbb{R}$ is a complete field.\bigskip

In the following, we delve into the case of our interest, that is when the continued fractions are obtained starting from an element $\alpha$. In order to properly speak about real and $p$--adic convergence to $\alpha$ of a continued fraction, $\alpha$ must be embedded in both the reals and the $p$--adics. Therefore we introduce some context and notation for the important cases of our study, the quadratic irrationals.\smallskip

A quadratic irrational is a root of an irreducible polynomial \[f(x)=ax^2+bx+c\in\mathbb{Q}[x]\]
of degree $2$ with rational coefficients. In our analysis, we deal with polynomials having a root both in $\mathbb{R}$ and in $\mathbb{Q}_p$. This means that the discriminant $D=b^2-4ac$ is positive, so that the equation $f(x)=0$ is solvable in $\mathbb{R}$, and it is a quadratic residue modulo $p$, so that the equation $f(x)=0$ is solvable in $\mathbb{Q}_p$. We call $\alpha^{(r)}$ and $\overline{\alpha}^{(r)}$ its roots in $\mathbb{R}$ and we call $\alpha^{(p)}$ and $\overline{\alpha}^{(p)}$ its roots in $\mathbb{Q}_p$.\smallskip

Capuano, Veneziano and Zannier \cite{CVZ} proved the following result for Ruban's continued fractions.

\begin{Theorem}[\cite{CVZ}, Proposition 4.5]\label{Thm: quadconv}
The Ruban's continued fraction of a $p$--adic quadratic irrational $\alpha^{(p)}$ always converges in $\mathbb{R}$ to some real number $\xi$.
\end{Theorem}

Ruban's continued fractions are not always periodic for quadratic irrationals. The periodic expansion of a $p$--adic quadratic irrational $\alpha^{(p)}$, with a technique similar to Proposition \ref{Pro: Conv}, can be showed to converge in $\mathbb{R}$ to either $\alpha^{(r)}$ or $\overline{\alpha}^{(r)}$. However, when Ruban's continued fraction of $\alpha^{(p)}$ is non-periodic, we know that it converges in $\mathbb{R}$, but we do not know to which real number. The authors of \cite{CVZ} left the following conjecture.

\begin{Conjecture}[\cite{CVZ}, Section 4.2]\label{Conj: cvz}
Ruban's continued fraction of a quadratic irrational is either periodic or it converges in $\mathbb{R}$ to a transcendental number.
\end{Conjecture}

This conjecture is quite reasonable, since algebraic numbers are countable and hence have measure $0$ in $\mathbb{R}$. Therefore, we expect a continued fraction having an  ``unknown" limit to converge to a transcendental number.\bigskip

For Browkin's continued fractions, we have experimentally observed similar results, for quadratic irrationals having a real image (see Section \ref{Sec: computations}). In the following we speak of either \textit{Browkin I} or \textit{Browkin II} or Algorithm \eqref{Alg: MR}, since they show similar behaviours of real convergence. We leave the following conjecture, that is similar to Theorem \ref{Thm: quadconv} for \textit{Browkin-type} continued fractions.

\begin{Conjecture}\label{Conj: BroConv}
\textit{Browkin I}, \textit{Browkin II} or Algorithm \eqref{Alg: MR} $p$-adic continued fractions of quadratic irrationals that can embedded into real numbers always converge in $\mathbb{R}$.
\end{Conjecture}

In fact, it computationally appears that Browkin's and Algorithm \eqref{Alg: MR} continued fractions converge, both in the periodic and the non-periodic cases, to some real number (see computations in Section \ref{Sec: computations}). The problem in generalizing the proof of Theorem \ref{Thm: quadconv} is again the presence of negative partial quotients.

\begin{Remark}\label{Rem: referee}
Notice the difference between the statement of Theorem \ref{Thm: quadconv} and our Conjecture \ref{Conj: BroConv}. In fact, in Theorem \ref{Thm: quadconv} the real convergence is proved for all quadratic irrationalities, also in the case where they can not be embedded in the reals, i.e. they have negative discriminant. This implies, in particular, that imaginary quadratic irrationals can not have a periodic Ruban's continued fraction. Instead, Conjecture \ref{Conj: BroConv} is stated only for quadratic irrationalities having positive discriminant. In fact, it is false for imaginary quadratic irrationals as, for example,
\begin{equation*}
\sqrt{-6}=\left[2,-\frac{7}{5},\overline{\frac{8}{5},-\frac{12}{5}}\right],
\end{equation*}
has a periodic \textit{Browkin I} continued fraction in $\mathbb{Q}_5$ and, hence, can not converge to a real number. On the contrary, \textit{Browkin I} continued fractions of imaginary quadratic irrationals for which periodicity is not observed, seem to converge to a real number.
\end{Remark}

Now we prove Proposition \ref{Pro: Conv}, stating that, in the periodic case, the $p$--adic continued fraction converges to the ``same" quadratic irrational in $\mathbb{R}$. In fact, we show that periodic $p$--adic continued fractions of quadratic irrationals that have embeddings both in $\mathbb{R}$ and in $\mathbb{Q}_p$, converge simultaneously in $\mathbb{R}$ and in $\mathbb{Q}_p$ to roots of the same minimal polynomial. Before proving Proposition \ref{Pro: Conv}, we recall the following classical and very general result for the convergence of periodic complex continued fractions in the Euclidean setting. The proof can be found in Wall's book \cite{WALL} or in Lane's original paper \cite{LAN}.

\begin{Theorem}[\cite{WALL}, Theorem 8.1]\label{Thm: Wall}
Let $[\overline{a_0,\ldots,a_{k-1}}]$ be a periodic continued fraction with complex partial quotients and let us call $\alpha_1$ and $\alpha_2$ the two roots of the polynomial relation of degree $2$ obtained by the periodicity. Then the continued fraction converges in $\mathbb{C}$ if and only if $\alpha_1$ and $\alpha_2$ are finite and one of the following conditions holds:
\begin{enumerate}
    \item[i)] $\alpha_1=\alpha_2$,
    \item[ii)] $\left|\alpha_2-\frac{A_{k-1}}{B_{k-1}}\right|>\left|\alpha_1-\frac{A_{k-1}}{B_{k-1}}\right|$, and $\frac{A_n}{B_n}\neq \alpha_2$ for $n=0,\ldots,k-2$.
\end{enumerate}
\end{Theorem}

We now prove Proposition \ref{Pro: Conv} by verifying that Condition $ii)$ of Theorem \ref{Thm: Wall} is satisfied by \textit{Browkin I}, \textit{Browkin II} and Algorithm \eqref{Alg: MR} $p$--adic continued fractions. Alternatively, it would be possible to use a technique similar to \cite[Theorem 4.3]{BEJ}.

\begin{proof}[Proof of Proposition \ref{Pro: Conv}]
Let us consider the periodic $p$--adic continued fraction
\[[a_0,\ldots,a_{h-1},\overline{a_h,\ldots,a_{h+k-1}}],\] obtained for the quadratic irrational $\alpha^{(p)}$ by means of either \textit{Browkin I}, \textit{Browkin II} or Algorithm \eqref{Alg: MR}. We know that $[a_0,a_1,\ldots]$ converges to $\alpha^{(p)}$ in $\mathbb{Q}_p$ by the correctness of the algorithms. Now we show that its periodic part satisfies the hypothesis $ii)$ of Theorem \ref{Thm: Wall} and, hence, that $[a_0,a_1,\ldots]$ converges to either $\alpha^{(r)}$ or $\overline{\alpha}^{(r)}$ in $\mathbb{R}$. In fact, the convergence is clearly not affected by the presence of the pre-period, and the periodicity gives raise to the same minimal polynomial of $\alpha^{(p)}$. Therefore, in the following we consider for simplicity $h=0$, i.e. the purely periodic continued fraction $[\overline{a_0,\ldots,a_{k-1}}]$. Notice that in order to satisfy condition $ii)$, since all the partial quotients are real numbers, it is sufficient to prove that it never happens that
\begin{equation}\label{Eq: FINAL}
\alpha_2^{(r)}-\frac{A_{k-1}}{B_{k-1}}=-\alpha_1^{(r)}+\frac{A_{k-1}}{B_{k-1}},
\end{equation}
where $\alpha_1^{(r)},\alpha_2^{(r)}$ are the two distinct real roots of
\begin{equation}\label{Eq: charfina}
B_{k-1}x^2+(B_{k-2}-A_{k-1})x-A_{k-2}=0.
\end{equation}

Recalling that
\[\alpha_1^{(r)}+\alpha_2^{(r)}=\frac{A_{k-1}-B_{k-2}}{B_{k-1}},\]
we get that Condition \eqref{Eq: FINAL} is equivalent to
\[\frac{2A_{k-1}}{B_{k-1}}=\alpha_1^{(r)}+\alpha_2^{(r)}=\frac{A_{k-1}-B_{k-2}}{B_{k-1}},\]
that can be rewritten as
\[\frac{A_{k-1}+B_{k-2}}{B_{k+1}}=0.\]
Therefore, \eqref{Eq: FINAL} holds if and only if 
\begin{equation}\label{Eq: finally}
A_{k-1}+B_{k-2}=0.
\end{equation}
Notice that if the latter holds then $k\neq 1$, as otherwise it would imply 
\[a_0=A_0=-B_{-1}=0,\]
that would correspond to the purely periodic continued fraction $[\overline{0}]$.
By hypothesis $\alpha_1^{(r)}$ and $\alpha_2^{(r)}$ are real and distinct, so that \eqref{Eq: charfina} has positive discriminant, i.e.
\[D=(B_{k-2}-A_{k-1})^2+4A_{k-2}B_{k-1}>0.\]
Rearranging the latter expression and exploiting \eqref{Eq: converelation} and \eqref{Eq: finally}, we obtain that $ii)$ is not satisfied if and only if
\begin{align*}
D&=(B_{k-2}-A_{k-1})^2+4A_{k-2}B_{k-1}=\\&=(B_{k-2}+A_{k-1})^2-4A_{k-1}B_{k-2}+4A_{k-2}B_{k-1}=4\cdot(-1)^{k-1}>0.
\end{align*}
Hence, whenever $D>0$, \eqref{Eq: FINAL} can hold only if the period $k$ is odd. The period length $k$ is always even for \textit{Browkin II}, therefore condition $ii)$ of Theorem \ref{Thm: Wall} is always satisfied for this algorithm. Moreover notice that, if the period length is odd in Algorithm \eqref{Alg: MR}, then all its partial quotients have negative valuations. The reason is that all the odd partial quotients in Algorithm \eqref{Alg: MR} have negative valuation, therefore if $v_p(a_i)=0$ for some even $i$, then by the periodicity also $a_{i+k}=a_i$ has zero valuation, that is not possible as $i+k$ is an odd index. Now, observe that $A_{k-1}=-B_{k-2}$ implies $v_p(A_{k-1})=v_p(B_{k-2})$. It is known that
\begin{align*}
v_p(A_{k-1})&=v_p(a_0)+v_p(a_1)+\ldots+v_p(a_{k-1}),\\
v_p(B_{k-2})&=v_p(a_1)+v_p(a_2)+\ldots+v_p(a_{k-2}).
\end{align*}
For the latter equalities see, for example, \cite[Lemma 1]{BI} and \cite[Remark 4.1]{R}.
It follows that $v_p(A_{k-1})=v_p(B_{k-2})$ only if
\begin{equation}
v_p(a_0)+v_p(a_{k-1})=0.
\end{equation}
This equation can not be satisfied by \textit{Browkin I} as, by construction, $v_p(a_0)\leq0$ and $v_p(a_{k-1})<0$, and we have seen that $k\neq1$ if \eqref{Eq: finally} holds. Therefore, \eqref{Eq: FINAL} is never satisfied for this algorithm. It remains to prove that the hypothesis $ii)$ of Theorem \ref{Thm: Wall} holds for Algorithm \eqref{Alg: MR}. We proved that Equation \eqref{Eq: FINAL} implies that the period length $k$ is odd, in the case of $D>0$. However, we showed also that, for odd $k$, all partial quotients in Algorithm \eqref{Alg: MR} have negative valuation. This implies that $v_p(a_0)+v_p(a_{k-1})<0$, so that $A_{k-1}\neq -B_{k-2}$, as we have seen before for \textit{Browkin I}. Therefore, \eqref{Eq: FINAL} can never be satisfied for \textit{Browkin I}, \textit{Browkin II} and Algorithm \eqref{Alg: MR}. It follows that the condition $ii)$ of Theorem \ref{Thm: Wall} is always satisfied for the $p$-adic continued fractions generated by these algorithms, and this proves the claim.
\end{proof}

From the computational results leading to Conjecture \ref{Conj: BroConv}, we see that the $p$--adic continued fraction of a quadratic irrational seems to always converge to a real number. Moreover, from Proposition \ref{Pro: Conv}, when its $p$--adic continued fraction is periodic, it converges in $\mathbb{R}$ to a root of the same quadratic polynomial. However, in cases where periodicity is not recognized, the continued fraction seems  to converge very neatly ``somewhere else" in $\mathbb{R}$. This means that, for those $p$--adic quadratic irrationals $\alpha^{(p)}$ that we believe to be non-periodic (based on brute-force search of the periodicity for a large number of steps), the continued fraction is approaching really well a real number that is completely different from $\alpha^{(r)}$ or $\overline{\alpha}^{(r)}$. Since this is a necessary condition for the periodicity of the $p$--adic continued fraction of $\alpha^{(p)}$, by Proposition \ref{Pro: Conv}, proving the convergence to another real number would be an example of a non-periodic Browkin's continued fraction. However, we are not able to exhibit such an example. Even if the continued fractions that are apparently non-periodic seem to converge really well to another real number, it remains a (very convincing, but heuristic) computational observation. All the computational results are listed in Section \ref{Sec: computations}.\bigskip

Also, we did not observe any continued fraction of $\alpha^{(p)}$ that seems to converge in $\mathbb{R}$ to $\alpha^{(r)}$ or $\overline{\alpha}^{(r)}$ and does not become periodic. If it does not show a period, the continued fraction is converging really well but really far from $\alpha$ in Euclidean absolute value. Therefore, it seems likely that Conjecture \ref{Conj: realconv} is true, and solving it would provide a really nice characterization of periodic $p$--adic continued fractions.\smallskip

The difficult part of Conjecture \ref{Conj: realconv} is to prove that if the $p$-adic continued fraction of the quadratic irrational $\alpha^{(p)}$ converges in $\mathbb{R}$ to either $\alpha^{(r)}$ or $\overline{\alpha}^{(r)}$, then it becomes periodic at some point. We leave the following two conjectures, that are supported by experimental computations (part of them is reported in Section \ref{Sec: computations}).

\begin{Conjecture}\label{Conj: conv2}
Let $[a_0,a_1,\ldots]$, where $a_n\in\mathbb{Q}$ for all $n$. If $[a_0,a_1,\ldots]$ converge simultaneously in $\mathbb{R}$ and $\mathbb{Q}_p$, respectively to a real quadratic irrational $\alpha^{(r)}$ and its image $\alpha^{(p)}$ inside $\mathbb{Q}_p$, then the continued fraction $[a_0,a_1,\ldots]$ is periodic.
\end{Conjecture}

\begin{Conjecture}\label{Conj: conv3}
Let $[a_0,a_1,\ldots]$ the $p$--adic continued fraction, obtained with either \textit{Browkin I}, \textit{Browkin II} or Algorithm \eqref{Alg: MR}, of a quadratic irrational $\alpha^{(p)}$ that has an image $\alpha$ inside $\mathbb{R}$. If $[a_0,a_1,\ldots]$ converge to either $\alpha^{(r)}$ or $\overline{\alpha}^{(r)}$ in $\mathbb{R}$, then the continued fraction is periodic.
\end{Conjecture}

Conjecture \ref{Conj: conv2} implies Conjecture \ref{Conj: conv3}, as we are giving more restrictive hypotheses. In fact, in Conjecture \ref{Conj: conv3} we have a more precise shape of the rational numbers of the continued fractions, because they are ``legal" partial quotients of \textit{Browkin I}, \textit{Browkin II} or Algorithm \eqref{Alg: MR}.

\begin{Example}\label{Exa: convereal}
Let us consider $\sqrt{19}$ in $\mathbb{Q}_5$. The polynomial $x^2-19$ has a root, and hence two roots, in $\mathbb{Q}_5$. This is the first and simplest example of $\sqrt{D}$, with $D>0$ non-square integer that is a quadratic residue modulo $5$, such that \textit{Browkin I} and \textit{Browkin II} continued fractions seem to be non-periodic. The expansions of $\sqrt{14}$ and $\sqrt{21}$, that are the quadratic residues modulo $5$ nearest to $\sqrt{19}$, have expansions
\begin{align*}
\sqrt{14}=&\left[2,-\frac{3}{5},\overline{-\frac{9}{5},-\frac{6}{5},\frac{166}{125},-\frac{6}{5},-\frac{9}{5},-\frac{8}{5}}\right],\\
\sqrt{21}=&\left[1,\frac{3}{5},\overline{\frac{3}{5},-\frac{4}{5},\frac{3}{5},-\frac{7}{5},\frac{26}{25},-\frac{7}{5}}\right],
\end{align*}
with \textit{Browkin I} and
\begin{align*}
\sqrt{14}=&\left[2,\overline{\frac{2}{5},-1,-\frac{1}{5},2,-\frac{1}{5},-1}\right],\\
\sqrt{21}=&\left[1,\overline{-\frac{2}{5},1,\frac{2}{5},2,\frac{1}{5},-2,\frac{2}{5},-2,\frac{1}{5},2,\frac{2}{5},1,-\frac{2}{5},2}\right],
\end{align*}
with \textit{Browkin II}. For $\sqrt{19}$, we did not find any period with a search up to $50.000$ complete quotients. It seems really unlikely to have periodicity of small square roots for small primes $p$ with such a long period. In fact, ``small" square roots tend to have short periods whenever they are periodic, as it happens for real continued fractions and as we can see for $\sqrt{14}$ and $\sqrt{21}$. This observation already suggests that is really unlikely for $\sqrt{19}$ to have a periodic continued fraction with \textit{Browkin I} and \textit{Browkin II}. Also, if we look at the convergents, we can notice the following. For the \textit{Browkin I} expansion of $\sqrt{19}$,
\begin{align*}
\frac{A_{10}}{B_{10}}&= 1.35736553571026\ldots,\\
\frac{A_{100}}{B_{100}}&=1.35738766711068\ldots,\\
\frac{A_{1000}}{B_{1000}}&=1.35738766711068\ldots,\\
\frac{A_{5000}}{B_{5000}}&=1.35738766711068\ldots,\\
\frac{A_{10000}}{B_{10000}}&=1.35738766711068\ldots,\\
\end{align*}
and this approximation just gets better going forward with the expansion. It seems that the \textit{Browkin I} continued fraction of $\sqrt{19}$ is converging to a real number which is, however, different from $\sqrt{19}\approx 4.36$. Proving that this expansion is not actually convergent to something near $4.36$ would be a counterexample for the periodicity of \textit{Browkin I}, by Proposition \ref{Pro: Conv}. However, we did not manage to prove it for $\sqrt{19}$ or for any other quadratic irrational. Similar results are true for \textit{Browkin II}, where the convergents of $\sqrt{19}$ are
\begin{align*}
\frac{A_{10}}{B_{10}}&= 2.57225268855495\ldots,\\
\frac{A_{100}}{B_{100}}&=1.89462102495469\ldots,\\
\frac{A_{1000}}{B_{1000}}&=1.89443989021177\ldots,\\
\frac{A_{5000}}{B_{5000}}&=1.89443989021177\ldots,\\
\frac{A_{10000}}{B_{10000}}&=1.89443989021177\ldots.\\
\end{align*}

Also \textit{Browkin II} convergents of $\sqrt{19}$ seem clearly converging to some real number. This real number, also in this case, is different from $\sqrt{19}\approx 4.36$. With our algorithm \eqref{Alg: MR}, $\sqrt{19}$ has periodic expansion
\[\sqrt{19}= \left[2,\overline{-\frac{2}{5},2,\frac{1}{5},-2,-\frac{2}{5},-\frac{12}{5},\frac{2}{5},-2,\frac{8}{25},2,\frac{1}{5},-1,-\frac{2}{5},-\frac{8}{5},\frac{2}{5},-2,\frac{12}{25},2,\frac{2}{5},-1} \right].
\]
For $\sqrt{39}$ we did not observed any period up to $50.000$ steps with Algorithm \eqref{Alg: MR}. As for Browkin's algorithms, the convergents seem to approach very neatly some real number. In fact,
\begin{align*}
\frac{A_{10}}{B_{10}}&= 3.11589405199643\ldots,\\
\frac{A_{100}}{B_{100}}&=3.24461192422490\ldots,\\
\frac{A_{1000}}{B_{1000}}&=3.23880830293096\ldots,\\
\frac{A_{5000}}{B_{5000}}&=3.23880830293096\ldots,\\
\frac{A_{10000}}{B_{10000}}&=3.23880830293096\ldots,\\
\end{align*}
that is far from $\sqrt{39}\approx 6.24$, to which they should converge if the continued fraction became periodic.
\end{Example}

It seems really likely, in Example \ref{Exa: convereal}, that the convergents of $\sqrt{19}$ and $\sqrt{39}$ are converging to something else rather than to the real value of $\sqrt{19}$ and $\sqrt{39}$. We can notice it also by looking at the difference between two consecutive partial quotients
\begin{equation}\label{Eq: diffconvergents}
\left|\frac{A_{n+1}}{B_{n+1}}-\frac{A_n}{B_n}\right|=\frac{1}{\left|B_n\right|\left|B_{n+1}\right|}.
\end{equation}

The denominators of the convergents are generated as in \eqref{Eq: Recursions}, therefore
\[B_{n+1}=a_{n+1}B_{n}+B_{n-1}.\]

For the continued fractions provided by Browkin's algorithms and Algorithm \eqref{Alg: MR}, the denominators $B_n$ are not always increasing. In fact, since the partial quotients are allowed to be negative, and they can be in general arbitrarily small, the sequence $B_n$ is not necessarily increasing. However, as it is true that $|B_n|$ can also be small, it does not usually happen.\smallskip

In fact, in order to have a small partial quotient, for example equal to $\frac{1}{p^5}$ in \textit{Browkin I}, at some step $n$, we should have

\[v_p(\alpha_{n-1}-a_{n-1})=v_p(c_1p+c_2p^2+\ldots)=5.\]

It means that $c_1=c_2=c_3=c_4=0$, where $c_1,c_2,c_3,c_4\in\left\{-\frac{p-1}{2},\ldots,\frac{p-1}{2} \right\}$ are representatives modulo $p$. However, as explained in the next section, it is widely believed that the $p$-adic digits of irrational numbers are uniformly distributed in $\{0,\ldots,p-1\}$, and hence in $\left\{-\frac{p-1}{2},\ldots,
\frac{p-1}{2} \right\}$. Therefore, the probability of having four consecutive $0$'s is $\frac{1}{p^4}$. The dominating term of $B_n$ is $a_1\cdot a_2\cdot\ldots\cdot a_n$. Therefore, when $|a_i|>1$, we expect to have an exponential growth of the $B_n$. Also, especially for big values of $p$, it is very rare to have $|a_i|<1$. In particular, notice that if in the $p$-adic expansion of
\[a_i=c_{-r}\frac{1}{p^r}+\ldots+c_{-1}\frac{1}{p}+c_0,\]
$|c_0|\geq 2$, then $|a_i|>1$.\smallskip

This observation  is the reason why, in the next section, we carry out a probabilistic argument. 

\section{A probabilistic approach to the real convergence}\label{Sec: prob}
In this section, we use Assumption \ref{Ass: equi} of uniform distribution of the digits in the $p$-adic expansion of a quadratic irrational. The aim is to build a probabilistic argument for the non-periodicity of $p$--adic continued fractions obtained with \textit{Browkin I}, \textit{Browkin II} and Algorithm \eqref{Alg: MR}. As already anticipated in the introduction, Assumption \ref{Ass: equi} is based on Borel's normal number theorem \cite{BOR}.

\begin{Definition}
Let $\alpha\in\mathbb{R}$ and let $b\geq 2$ be an integer. Then $\alpha$ is normal in base $b$ if every digit of the expansion of $\alpha$ in base $b$ has density $\frac{1}{b}$. We say that $\alpha$ is normal, or absolutely normal, if it is normal in every base $b\geq 2$.
\end{Definition}

In \cite{BOR}, Borel proved the following result.

\begin{Theorem}[Borel's Theorem]\label{Thm: Borel}
Almost all real numbers are normal, i.e. the set of non-normal numbers has Lebesgue measure $0$ in $\mathbb{R}$.
\end{Theorem}

Borel's Theorem is proved by applying Borel-Cantelli Lemma. It is known that the strategy of the proof can be applied equivalently to the $p$--adic expansion of $p$--adic numbers. Therefore, almost all $p$--adic numbers are normal, i.e. every digit in $\mathbb{Z}/p\mathbb{Z}$ has density $\frac{1}{p}$. In particular, all irrational numbers, both in $\mathbb{R}$ and in $\mathbb{Q}_p$, are conjectured to be normal. However, apart from Borel's existence result, only few numbers have been proved to be normal. For example, there is not a proof that $\sqrt{2}$, $\pi$ or $e$ are normal, although it is strongly considered true.

\begin{Remark}
By Assumption \ref{Ass: equi}, it is implied also the uniform distribution of the coefficients of the partial quotients in both Browkin's and Ruban's continued fractions (see, e.g., the metric results in \cite{RUB}).
\end{Remark}

In the following, for sake of simplicity, we focus on \textit{Browkin I} algorithm. However, the same arguments can be adapted to obtain similar results also for \textit{Browkin II} and Algorithm \eqref{Alg: MR}.\smallskip

First of all, we start by computing the probability for a complete quotient of the \textit{Browkin I} continued fraction to have a given $p$--adic valuation.

\begin{Proposition}\label{Pro: probval}
Let $\alpha_0\in\mathbb{Q}_p$ be a quadratic irrational. Then
\begin{equation}\label{Eq: distribution}
\mathbb{P}(v_p(\alpha_n)=-k)=\frac{p-1}{p^k},
\end{equation}
for all $n,k\geq 1$.
\begin{proof}
The valuation of $\alpha_{n}$ is $-k$ if and only if
\[\alpha_{n-1}-a_{n-1}=c_kp^k+\ldots.\]
This is equivalent to ask that $c_1=c_2=\ldots=c_{k-1}=0$, that happens with probability
\[\mathbb{P}(c_1=c_2=\ldots=c_{k-1}=0)=\frac{1}{p^{k-1}},\]
and $c_k\neq 0$, that happens with probability
\[\mathbb{P}(c_k\neq 0)=\frac{p-1}{p},\]
and this proves the claim.
\end{proof}
\end{Proposition}

\begin{Remark}
Notice that, in fact
\begin{align*}
\sum\limits_{k=1}^{+\infty}\frac{p-1}{p^k}&=(p-1)\left(\sum\limits_{k=1}^{+\infty}\frac{1}{p^k}\right)=(p-1)\left(\sum\limits_{k=0}^{+\infty}\frac{1}{p^k}\right)-(p-1)=\\
&=(p-1)\left(\frac{1}{1-\frac{1}{p}}\right)-(p-1)=p-(p-1)=1,
\end{align*}
so that \eqref{Eq: distribution} actually defines a distribution on the natural numbers.
\end{Remark}

Then, in the next proposition we compute the expected value of the valuation of a \textit{Browkin I} complete quotient.

\begin{Proposition}\label{Prop: MeanVal}
Let $\alpha_0\in\mathbb{Q}_p$ be a quadratic irrational. Then,
\[\mathbb{E}(v_p(\alpha_n))=\frac{p}{p-1},\]
for all $n\geq 1$.
\begin{proof}
Let us notice that
\[\mathbb{E}(v_p(\alpha_n))=\sum\limits_{k=1}^{+\infty}k\frac{p-1}{p^k},\]
is a convergent series since
\[\lim\limits_{k\rightarrow +\infty }\frac{(k+1)\frac{p-1}{p^{k+1}}}{k\frac{p-1}{p^{k}}}=\lim\limits_{k\rightarrow +\infty}\frac{k+1}{kp}=\frac{1}{p}<1.\]
Then we compute
\begin{align*}
\mathbb{E}(v_p(\alpha_n))&=\sum\limits_{k=1}^{+\infty}k\frac{p-1}{p^k}=(p-1)\sum\limits_{k=1}^{+\infty}\frac{k}{p^k}+(p-1)\sum\limits_{k=1}^{+\infty}\frac{1}{p^k}-(p-1)\sum\limits_{k=1}^{+\infty}\frac{1}{p^k}=\\
&=(p-1)\sum\limits_{k=1}^{+\infty}\frac{k-1}{p^k}+(p-1)\sum\limits_{k=1}^{+\infty}\frac{1}{p^k}=\\
&=(p-1)\sum\limits_{k=0}^{+\infty}\frac{k}{p^{k+1}}+(p-1)\left(\frac{1}{1-\frac{1}{p}}-1\right)=\\
&=\frac{1}{p}\sum\limits_{k=0}^{+\infty}k\frac{p-1}{p^k}+1=\frac{1}{p}\sum\limits_{k=1}^{+\infty}k\frac{p-1}{p^k}+1.
\end{align*}
It means that, for $\mathbb{E}=\mathbb{E}(v_p(\alpha_n))$,
\begin{equation*}
\mathbb{E}=\frac{1}{p}\mathbb{E}+1,
\end{equation*}
that is,
\[\mathbb{E}=\mathbb{E}(v_p(\alpha_n))=\frac{p}{p-1},\]
and the claim is proved.
\end{proof}
\end{Proposition}

\begin{Remark}
Proposition \ref{Prop: MeanVal} tells us that, especially for large prime $p$, having large negative valuation is extremely rare. This is what we expect, since large valuations are caused by several consecutive $0's$, and under Assumption \ref{Ass: equi} each zero coefficient appears with probability $\frac{1}{p}$.
\end{Remark}

Now we want to compute the expected size of \textit{Browkin I} partial quotients inside $\mathbb{R}$, in order to predict the behaviour of the Euclidean absolute value of the convergents. Notice that the expected value of any partial quotient $a$ in \textit{Browkin I} is $\mathbb{E}(a)=0$. This is due to the choice of the symmetric interval $\{-\frac{p-1}{2},\ldots,\frac{p-1}{2}\}$. Therefore, to have an estimate of the size of the partial quotients, we compute the expected value of its Euclidean absolute value.

\begin{Theorem}\label{Thm: expvalue}
Let $a$ be a \textit{Browkin I} partial quotient and let $k\geq 1$. Then, the expected value of its Euclidean absolute value  is
\[\mathbb{E}(|a|\ |\ v_p(a)=-k)=\frac{p^{2(k+1)}-1}{4p^{2k+1}},\]
and, in general,
\[\mathbb{E}(|a|)=\frac{p}{4}\left(1-\frac{1}{p^2(p^2+p+1)}\right).\]
\begin{proof}
The expected value of the constant term $c_0$ is $0$, but in absolute value it is
\begin{align*}
\mathbb{E}(|c_0|)&=\frac{1}{p}(0+2+\ldots+p-1)=\frac{2}{p}\left(1+\ldots+\frac{p-1}{2}\right)=\\
&=\frac{2}{p}\left(\frac{p-1}{2}\right)\left(\frac{p+1}{2}\right)\frac{1}{2}=\frac{p^2-1}{4p}.
\end{align*}
The expected value of $|c_0+c_1\frac{1}{p}|$ is composed by the case where $c_0=0$, i.e.
\begin{align*}
\frac{2}{p}\left(1+\ldots+\frac{p-1}{2}\right),
\end{align*}
and the case where $c_0\neq 0$. When $c_0\neq 0$, the positive elements $c_0+c_1\frac{1}{p}$ are exactly the ones with positive $c_0\in\left\{0,\ldots,\frac{p-1}{2}\right\}$. For $c_0\neq 0$, all the elements $c_0+c_1p$ and $c_0-c_1p$ cancel out, and all of them are counted twice (the positive and the negative one). Therefore, for all choices of $c_0\neq 0$ we have a summand of the kind $2pc_0$, since the tails with $c_1$ eliminate and there are exactly $p$ summands. Now we can compute
\begin{align*}
\mathbb{E}\left(\left|c_0+\frac{c_1}{p}\right|\right)&=\frac{1}{p^2}\left(\frac{2}{p}\left(1+\ldots+\frac{p-1}{2}\right)+2p\left(1+\ldots+\frac{p-1}{2}\right)\right)=\\
&=\frac{1}{p^2}\cdot\frac{p-1}{2}\cdot\frac{p+1}{2}\cdot\frac{1}{2}\left(\frac{2}{p}+2p\right)=\\
&=\frac{p^2-1}{8p^2}\cdot\frac{2+2p^2}{p}=\frac{p^4-1}{4p^3}.
\end{align*}
In the general case, with a similar reasoning we obtain
\begin{align*}
\mathbb{E}\left(\left|c_0+\frac{c_1}{p}+\ldots+\frac{c_k}{p^k}\right|\right)&=\frac{1}{p^2}\mathbb{E}\left(\left|c_1+\ldots+\frac{c_k}{p^{k-1}}\right|\right)+\frac{2p^k}{p^{k+1}}\left(1+\ldots+\frac{p-1}{2}\right)=\\
&=\frac{1}{p^2}\mathbb{E}\left(\left|c_1+\ldots+\frac{c_k}{p^{k-1}}\right|\right)+\frac{2p^k(p^2-1)}{8p^{k+1}}=\\
&=\frac{1}{p^2}\mathbb{E}\left(\left|c_1+\ldots+\frac{c_k}{p^{k-1}}\right|\right)+\frac{p^2-1}{4p}.
\end{align*}
The reason why $\frac{1}{p^2}$ appears in the first summand, is that, when $c_0=0$, we can collect $\frac{1}{p}$ from the sum and, moreover, the expected value on $k+1$ terms is divided by $p$ one time more than the expected value on $k$ terms. For the second summand, as before we have two times every value, every $k$-tuple cancel and they are exactly $p^k$, so we multiply that for every value of $c_0$ in $\{0,\ldots,\frac{p-1}{2}\}$. Now, if we call
\[\mathbb{E}_k=\mathbb{E}\left(\left|c_0+\frac{c_1}{p}+\ldots+\frac{c_k}{p^k}\right|\right),\]
we have:
\begin{align*}
\mathbb{E}_k&=\frac{p^2-1}{4p}+\frac{1}{p^2}\mathbb{E}_{k-1}=\frac{p^2-1}{4p}+\frac{p^2-1}{4p}\cdot\frac{1}{p^2}+\frac{1}{p^4}\mathbb{E}_{k-2}=\\
&=\frac{p^2-1}{4p}\left( 1+\frac{1}{p^2}+\frac{1}{p^4}+\ldots+\frac{1}{p^{2(k-1)}}\right)+\frac{1}{p^{2k}}\mathbb{E}_0=\\
&=\frac{p^2-1}{4p}\left( 1+\frac{1}{p^2}+\frac{1}{p^4}+\ldots+\frac{1}{p^{2(k-1)}}+\frac{1}{p^{2k}}\right)=\\
&=\frac{p^2-1}{4p}\left(\frac{p^{2(k+1)}-1}{p^{2k}(p^2-1)}\right)=\frac{p^{2(k+1)}-1}{4p^{2k+1}}.
\end{align*}

Then, we can compute
\begin{align*}
\mathbb{E}(|a|)&=\sum\limits_{k=1}^{+\infty}\mathbb{E}(|a| \ | v_p(a)=-k)\mathbb{P}(v_p(a)=-k)=\sum\limits_{k=1}^{+\infty}\frac{p^{2(k+1)}-1}{4p^{2k+1}}\cdot\frac{p-1}{p^k}=\\
&=\frac{p-1}{4}\sum\limits_{k=1}^{+\infty}\frac{p^{2k+2}-1}{p^{3k+1}}=\frac{p-1}{4}\left(\sum\limits_{k=1}^{+\infty}\frac{1}{p^{k-1}}-\sum\limits_{k=1}^{+\infty}\frac{1}{p^{3k+1}}\right)=\\
&=\frac{p-1}{4}\left(\sum\limits_{k=0}^{+\infty}\frac{1}{p^{k}}-\frac{1}{p^4}\sum\limits_{k=0}^{+\infty}\frac{1}{p^{3k}}\right)=\frac{p-1}{4}\left(\frac{1}{1-\frac{1}{p}}-\frac{1}{p^4}\cdot\frac{1}{1-\frac{1}{p^3}}\right)=\\
&=\frac{p-1}{4}\left(\frac{p}{p-1}-\frac{1}{p}\cdot\frac{1}{p^3-1}\right)=\frac{1}{4}\left(p-\frac{1}{p}\cdot\frac{1}{p^2+p+1}\right)=\\
&=\frac{p}{4}\left(1-\frac{1}{p^2(p^2+p+1)}\right),
\end{align*}
and this concludes the proof.
\end{proof}
\end{Theorem}

\begin{Remark}
Let us notice that we always expect $|a|>1$ for all primes $p$ except $p=3$. Moreover, the expected size of every partial quotients grows linearly with $p$:
\begin{align*}
p=13 \ \ \ \ &\Longrightarrow \ \ \ \ \mathbb{E}(|a|)\approx 3.25,\\
p=43 \ \ \ \ &\Longrightarrow \ \ \ \ \mathbb{E}(|a|)\approx 10.75,\\
p=211 \ \ \ \ &\Longrightarrow \ \ \ \ \mathbb{E}(|a|)\approx 52.75,\\
p=839 \ \ \ \ &\Longrightarrow \ \ \ \ \mathbb{E}(|a|)\approx 209.75.
\end{align*}
\end{Remark}

\begin{Example}\label{Exa: deviation}
For $p=7823$, the convergents of the expansion of $\sqrt{15648}$ with \textit{Browkin I} are:
\begin{align*}
\frac{A_{10}}{B_{10}}&= 3339.99956244164\ldots,\\
\frac{A_{100}}{B_{100}}&=3339.99956244164\ldots,\\
\frac{A_{1000}}{B_{1000}}&=3339.99956244164\ldots,\\
\frac{A_{5000}}{B_{5000}}&=3339.99956244164\ldots,\\
\frac{A_{10000}}{B_{10000}}&=3339.99956244164\ldots.\\
\end{align*}
Indeed, the continued fraction seems to rapidly converge to a real limit, which is different from $\sqrt{15648}\approx 125.09$. In this case, it is really unlikely for the convergents to deviate from that value, due to Equation \eqref{Eq: diffconvergents}. In fact, we expect $|B_n|$ to grow exponentially as
\[|B_n|\sim |a_1|\ldots|a_n|,\]
where the expected value of each partial quotient is, by Theorem \ref{Thm: expvalue},
\[\mathbb{E}(|a|)\approx 1955.75.\]
\end{Example}

Finally, based on these results and the computational analysis performed in the next section we leave Conjecture \ref{Conj: realconv}. In practice, by Conjecture \ref{Conj: realconv}, even if cancellation can happen and the convergents can be arbitrarily oscillating, this is really unlikely, as the sequence $|B_n|$ tends to have an exponential growth.

\section{Computations}\label{Sec: computations}
In the following tables, we list the approximations of $\sqrt{D}$, having both a $p$-adic and a real image, provided by the $1000$th convergent of \textit{Browkin I}, \textit{Browkin II} and Algorithm \eqref{Alg: MR} $p$--adic continued fractions in $\mathbb{Q}_5$ and $\mathbb{Q}_{43}$, distinguishing the periodic and the apparently non-periodic cases. It is possible to notice that, when the $p$--adic continued fraction of $\sqrt{D}$ is periodic, the convergents approach very well and very soon the real limit, that is $\sqrt{D}$ by Proposition \ref{Pro: Conv}. In the cases where periodicity is not detected, the continued fraction seems to approach a real limit different from $\sqrt{D}$. This is a behaviour that we have observed in all the examples of our analysis. The SageMath code used for the computational part is publicly available\footnote{\href{https://github.com/giulianoromeont/p-adic-continued-fractions}{https://github.com/giulianoromeont/p-adic-continued-fractions}} and contains the implementations of the main algorithms for $p$--adic continued fractions.

\begin{table}[H]
 \begin{minipage}{0.5\linewidth}
  \caption{Non-periodic \textit{Browkin I} expansions within $1000$ steps in $\mathbb{Q}_5$.}
  \bigskip
  \centering
  \begin{tabular}{c|c}
   D & $(A_{1000}/B_{1000})^2$\\
   \hline
   19 & \ 1.84226    \\ \hline 
26 & \ 0.42758    \\ \hline 
29 & \ 1.56850    \\ \hline 
31 & \ 3.85808    \\ \hline 
39 & \ 1.46700    \\ \hline 
41 & \ 3.17409    \\ \hline 
44 & \ 5.17107    \\ \hline 
46 & \ 0.23921    \\ \hline 
51 & \ 0.04092    \\ \hline 
56 & \ 0.00378   \\ \hline 
59 & \ 6.02309    \\ \hline 
61 & \ 1.76172    \\ \hline 
66 & \ 6.74804    \\ \hline 
71 & \ 1.83277    \\ \hline 
79 & \ 6.64195    \\ \hline 
84 & \ 2.60241    \\ \hline 
86 & \ 10.4167    \\ \hline 
89 & \ 15.3162    \\ \hline 
91 & \ 5.00953    \\ \hline 
96 & \ 0.16662    \\ \hline 
101 & \ 1.78169    \\ \hline 
106 & \ 2.94191    \\ \hline 
109 & \ 7.64688    \\ \hline 
114 & \ 14.9227    \\ \hline 
116 & \ 0.10316    \\ \hline 
124 & \ 0.87815    \\ \hline 
126 & \ 1.38815    \\ \hline 
129 & \ 1.67262    \\ \hline
131 & \ 0.55175 \\ \hline 
134 & \ 7.69174 \\ \hline 
136 & \ 0.63808 \\ \hline 
139 & \ 1.46264 \\ \hline 
141 & \ 6.41153 \\ \hline 
146 & \ 2.86117 \\ \hline 
149 & \ 2.85846 \\ \hline 
151 & \ 0.00054 
  \end{tabular}
 \end{minipage}%
 \begin{minipage}{0.5\linewidth}
  \caption{Periodic \textit{Browkin I} expansions within $1000$ steps in $\mathbb{Q}_5$.}
  \bigskip
  \centering
  \begin{tabular}{c|c}
   D & $(A_{1000}/B_{1000})^2$ \\
   \hline
   6 & \ 6.00005\ \\ \hline 
11 & \ 11.00049  \\ \hline 
14 & \ 14.00031  \\ \hline 
21 & \ 20.99472 \\ \hline 
24 & \ 24.00020 \\ \hline 
34 & \ 34.00056 \\ \hline 
54 & \ 54.00045  \\ \hline 
69 & \ 68.99960  \\ \hline 
74 & \ 73.99268  \\ \hline 
76 & \ 75.99829  \\ \hline 
94 & \ 93.99884  \\ \hline 
99 & \ 99.00051  \\ \hline 
104 & \ 104.00124 \\ \hline 
111 & \ 110.99886 \\ \hline 
119 & \ 119.00191 \\  
  \end{tabular}
 \end{minipage}%
\end{table}

\newpage

\begin{table}[H]
 \begin{minipage}{0.5\linewidth}
  \caption{Non-periodic \textit{Browkin II} expansions within $1000$ steps in $\mathbb{Q}_5$.}
  \bigskip
  \centering
  \begin{tabular}{c|c}
   D & $(A_{1000}/B_{1000})^2$\\
   \hline
   19 & \ 3.58875  \\ \hline 
39 & \ 11.61105  \\ \hline 
41 & \ 0.24482  \\ \hline 
44 & \ 1.84987  \\ \hline 
46 & \ 0.18939  \\ \hline 
59 & \ 0.08105  \\ \hline 
66 & \ 3.49091  \\ \hline 
71 & \ 0.00555  \\ \hline 
74 & \ 5.30380  \\ \hline 
76 & \ 2.88456  \\ \hline 
86 & \ 0.13089  \\ \hline 
89 & \ 29.0251  \\ \hline 
94 & \ 13.4146  \\ \hline 
96 & \ 7.01349  \\ \hline 
99 & \ 0.03798  \\ \hline 
101 & \ 2.46992  \\ \hline 
106 & \ 0.00040 \\ \hline 
119 & \ 4.90356  \\ \hline 
124 & \ 18.46678  \\ \hline 
131 & \ 3.90023  \\ \hline 
134 & \ 7.01402  \\ \hline 
141 & \ 3.03630  \\ \hline 
146 & \ 0.00366  \\ \hline 
151 & \ 5.06295  \\ \hline 
154 & \ 18.45303  \\ \hline 
159 & \ 14.69725  \\ \hline
164 & \ 0.91049  \\ \hline 
166 & \ 0.42797  \\ \hline 
174 & \ 6.30210  \\ \hline 
179 & \ 0.06295 \\ \hline 
181 & \ 0.52794  \\ \hline 
184 & \ 2.80529  \\ \hline 
186 & \ 14.3330   \\ \hline 
189 & \ 2.5363  \\ \hline 
191 & \ 26.1437   \\ \hline 
194 & \ 2.6614  \\ \hline 
199 & \ 17.37722   
  \end{tabular}
 \end{minipage}%
 \begin{minipage}{0.5\linewidth}
  \caption{Periodic \textit{Browkin II} expansions within $1000$ steps in $\mathbb{Q}_5$.}
  \bigskip
  \centering
  \begin{tabular}{c|c}
   D & $(A_{1000}/B_{1000})^2$ \\
   \hline
   6 & \ 5.99956  \\ \hline 
11 & \ 10.99983  \\ \hline 
14 & \ 13.99957  \\ \hline 
21 & \ 21.92018  \\ \hline 
24 & \ 24.00020  \\ \hline 
26 & \ 25.99980  \\ \hline 
29 & \ 28.99930  \\ \hline 
31 & \ 30.99928  \\ \hline 
34 & \ 34.00056  \\ \hline 
51 & \ 50.99959  \\ \hline 
54 & \ 53.99898  \\ \hline 
56 & \ 55.99977  \\ \hline 
61 & \ 61.00078  \\ \hline 
69 & \ 38.92137  \\ \hline 
79 & \ 78.99832  \\ \hline 
84 & \ 84.75043  \\ \hline 
91 & \ 91.00015  \\ \hline 
104 & \ 104.0012  \\ \hline 
109 & \ 72.80696  \\ \hline 
111 & \ 110.9525  \\ \hline 
114 & \ 113.9983  \\ \hline 
116 & \ 115.9993  \\ \hline 
126 & \ 125.9983  \\ \hline 
129 & \ 128.9996  \\ \hline 
136 & \ 136.0022  \\ \hline 
139 & \ 138.9993  \\ \hline 
149 & \ 148.9888  \\ \hline 
156 & \ 155.9976  \\ \hline
161 & \ 160.8153 \\ \hline 
171 & \ 170.9974 \\ \hline 
176 & \ 175.9973 \\  
  \end{tabular}
 \end{minipage}%
\end{table}

\newpage

\begin{table}[H]
 \begin{minipage}{0.5\linewidth}
  \caption{Non-periodic Algorithm \eqref{Alg: MR} expansions within $1000$ steps in $\mathbb{Q}_5$.}
  \bigskip
  \centering
  \begin{tabular}{c|c}
   D & $(A_{1000}/B_{1000})^2$\\
   \hline
39 & \ 10.4898    \\ \hline 
41 & \ 0.17799    \\ \hline 
46 & \ 0.12453    \\ \hline 
66 & \ 0.07997    \\ \hline 
69 & \ 2.10163    \\ \hline 
71 & \ 0.00543    \\ \hline 
79 & \ 1.61061    \\ \hline 
86 & \ 0.13053   \\ \hline 
89 & \ 13.0321    \\ \hline 
101 & \ 2.4583   \\ \hline 
106 & \ 4968.642 \\ \hline 
109 & \ 85110.827 \\ \hline 
114 & \ 26.4699   \\ \hline 
124 & \ 18.35951   \\ \hline 
131 & \ 3.6963   \\ \hline 
136 & \ 8.6183   \\ \hline 
141 & \ 4.7646   \\ \hline 
146 & \ 1.04080   \\ \hline 
149 & \ 0.02709   \\ \hline 
151 & \ 2.85170    \\ \hline 
154 & \ 19.7118    \\ \hline 
159 & \ 16.6455    \\ \hline 
161 & \ 10.9441    \\ \hline 
164 & \ 0.7972    \\ \hline 
171 & \ 0.0039    \\ \hline 
174 & \ 19.1686    \\ \hline 
179 & \ 0.1908    \\ \hline 
181 & \ 31449.44 \\ \hline 
184 & \ 2.80562    \\ \hline 
186 & \ 22.7023    \\ \hline 
189 & \ 15.1905    \\ \hline 
194 & \ 2.6604    \\ \hline 
199 & \ 15.142    \\ 
  \end{tabular}
 \end{minipage}%
 \begin{minipage}{0.5\linewidth}
  \caption{Periodic Algorithm \eqref{Alg: MR} expansions within $1000$ steps in $\mathbb{Q}_5$.}
  \bigskip
  \centering
  \begin{tabular}{c|c}
   D & $(A_{1000}/B_{1000})^2$ \\
   \hline
   6 & \ 5.9995  \\ \hline 
11 & \ 10.9998  \\ \hline 
14 & \ 13.9995  \\ \hline 
19 & \ 19.000  \\ \hline 
21 & \ 21.9201  \\ \hline 
24 & \ 24.0002  \\ \hline 
26 & \ 25.999  \\ \hline 
29 & \ 28.9993  \\ \hline 
31 & \ 30.9992  \\ \hline 
34 & \ 34.0005  \\ \hline 
44 & \ 43.9900  \\ \hline 
51 & \ 50.9995  \\ \hline 
54 & \ 54.0004  \\ \hline 
56 & \ 55.9997  \\ \hline 
59 & \ 59.0008  \\ \hline 
61 & \ 61.0007  \\ \hline 
74 & \ 73.9995  \\ \hline 
76 & \ 75.9982  \\ \hline 
84 & \ 84.7504  \\ \hline 
91 & \ 90.9982  \\ \hline 
94 & \ 93.9988  \\ \hline 
96 & \ 95.9988  \\ \hline 
99 & \ 99.0005  \\ \hline 
104 & \ 104.001  \\ \hline 
111 & \ 110.996  \\ \hline 
116 & \ 115.999  \\ \hline 
119 & \ 117.887  \\ \hline 
126 & \ 125.998 \\ \hline 
129 & \ 129.001  \\ \hline 
134 & \ 133.999  \\ \hline 
139 & \ 138.999  \\ \hline 
156 & \ 155.997  \\ \hline 
166 & \ 142.515  \\ \hline 
176 & \ 174.195  \\ \hline 
191 & \ 190.997  \\
  \end{tabular}
 \end{minipage}%
\end{table}

\begin{table}[H]
 \begin{minipage}{0.5\linewidth}
  \caption{Non-periodic \textit{Browkin I} expansions within $1000$ steps in $\mathbb{Q}_{43}$.}
  \bigskip
  \centering
  \begin{tabular}{c|c}
   D & $(A_{1000}/B_{1000})^2$\\
   \hline
  6 & \ 44.7453  \\ \hline 
10 & \ 222.0129  \\ \hline 
11 & \ 436.6093  \\ \hline 
13 & \ 402.1308  \\ \hline 
14 & \ 98.7578  \\ \hline 
15 & \ 146.7199  \\ \hline 
17 & \ 367.3585  \\ \hline 
21 & \ 62.1133 \\ \hline 
23 & \ 326.9225  \\ \hline 
24 & \ 197.4165  \\ \hline 
31 & \ 296.631  \\ \hline 
38 & \ 82.5499  \\ \hline 
40 & \ 166.265  \\ \hline 
41 & \ 254.491  \\ \hline 
44 & \ 0.9057  \\ \hline 
47 & \ 4.3664  \\ \hline 
52 & \ 8.1533  \\ \hline 
53 & \ 227.689  \\ \hline 
54 & \ 438.270  \\ \hline 
56 & \ 403.029  \\ \hline 
57 & \ 98.7161  \\ \hline 
58 & \ 146.785  \\ \hline 
59 & \ 15.5031  \\ \hline 
60 & \ 356.042  \\ \hline 
66 & \ 330.428  \\ \hline 
67 & \ 192.820  \\ \hline 
68 & \ 26.185  \\ \hline 
74 & \ 286.845  \\ \hline 
79 & \ 36.648  \\ \hline 
83 & \ 174.448  \\ \hline 
84 & \ 247.882  \\ \hline 
87 & \ 0.907  \\ \hline 
90 & \ 4.373  \\ \hline 
92 & \ 43.991  \\ \hline 
95 & \ 8.1618  \\ \hline 
96 & \ 222.173  \\ \hline 
97 & \ 445.349  \\ \hline 
99 & \ 424.615  \\ 
  \end{tabular}
 \end{minipage}%
 \begin{minipage}{0.5\linewidth}
  \caption{Periodic \textit{Browkin I} expansions within $1000$ steps in $\mathbb{Q}_{43}$.}
  \bigskip
  \centering
  \begin{tabular}{c|c}
   D & $(A_{1000}/B_{1000})^2$ \\
   \hline
35 & \ 35.000  \\ \hline 
78 & \ 78.000 \\ \hline 
187 & \ 186.997  \\ 
  \end{tabular}
 \end{minipage}%
\end{table}

\newpage

\begin{table}[H]
 \begin{minipage}{0.5\linewidth}
  \caption{Non-periodic \textit{Browkin II} expansions within $1000$ steps in $\mathbb{Q}_{43}$.}
  \bigskip
  \centering
  \begin{tabular}{c|c}
   D & $(A_{1000}/B_{1000})^2$\\
   \hline
 10 & \ 115.362  \\ \hline 
11 & \ 672.676  \\ \hline 
13 & \ 331.167  \\ \hline 
15 & \ 57.450  \\ \hline 
23 & \ 1301.16 \\ \hline 
31 & \ 220.368  \\ \hline 
35 & \ 60.478 \\ \hline 
40 & \ 240.74\\ \hline 
53 & \ 300.467 \\ \hline 
54 & \ 358.488  \\ \hline 
56 & \ 569.003  \\ \hline 
58 & \ 84.0687  \\ \hline 
60 & \ 464.238  \\ \hline 
66 & \ 188.872  \\ \hline 
67 & \ 411.809  \\ \hline 
74 & \ 352.617  \\ \hline 
83 & \ 106.822  \\ \hline 
96 & \ 135.124  \\ \hline 
97 & \ 299.172  \\ \hline 
99 & \ 318.033 \\ \hline 
101 & \ 210.934  \\ \hline 
103 & \ 676.431  \\ \hline 
109 & \ 445.940  \\ \hline 
110 & \ 133.830  \\ \hline
117 & \ 389.825 \\ \hline 
126 & \ 224.244 \\ \hline 
127 & \ 325.751  \\ \hline 
130 & \ 9.8125  \\ \hline 
133 & \ 4.7093  \\ \hline 
139 & \ 162.889  \\ \hline 
142 & \ 492.067 \\ \hline 
145 & \ 34.729  \\ \hline 
146 & \ 1808.22 \\ \hline 
153 & \ 309.851  \\ \hline 
154 & \ 2.678\\ \hline 
160 & \ 1137.27
  \end{tabular}
 \end{minipage}%
 \begin{minipage}{0.5\linewidth}
  \caption{Periodic \textit{Browkin II} expansions within $1000$ steps in $\mathbb{Q}_{43}$.}
  \bigskip
  \centering
  \begin{tabular}{c|c}
   D & $(A_{1000}/B_{1000})^2$ \\
   \hline
6 & \ 5.9995  \\ \hline 
14 & \ 13.999  \\ \hline 
17 & \ 17.006 \\ \hline 
21 & \ 20.999  \\ \hline 
24 & \ 23.999  \\ \hline 
38 & \ 37.999  \\ \hline 
41 & \ 40.999  \\ \hline 
44 & \ 43.999 \\ \hline 
47 & \ 46.999  \\ \hline 
52 & \ 51.999  \\ \hline 
57 & \ 56.999  \\ \hline 
59 & \ 58.999 \\ \hline 
68 & \ 67.999  \\ \hline 
78 & \ 77.998  \\ \hline 
79 & \ 78.998  \\ \hline 
84 & \ 83.999  \\ \hline 
87 & \ 86.998  \\ \hline 
90 & \ 89.999  \\ \hline 
92 & \ 91.998 \\ \hline 
95 & \ 94.998  \\ \hline 
102 & \ 101.999  \\ \hline 
107 & \ 106.998  \\ \hline 
111 & \ 110.998  \\ \hline 
122 & \ 121.998  \\ \hline 
124 & \ 123.999  \\ \hline 
135 & \ 134.998 \\ \hline 
138 & \ 137.999 \\ \hline 
140 & \ 139.998  \\ \hline 
143 & \ 142.998 \\ \hline 
150 & \ 149.998  \\ \hline 
152 & \ 151.999 \\
  \end{tabular}
 \end{minipage}%
\end{table}

\newpage

\begin{table}[H]
 \begin{minipage}{0.5\linewidth}
  \caption{Non-periodic Algorithm \eqref{Alg: MR} expansions within $1000$ steps in $\mathbb{Q}_{43}$.}
  \bigskip
  \centering
  \begin{tabular}{c|c}
   D & $(A_{1000}/B_{1000})^2$\\
   \hline
10 & \ 101.090 \\ \hline 
11 & \ 675.849  \\ \hline 
13 & \ 331.167 \\ \hline 
15 & \ 57.4503 \\ \hline 
23 & \ 1301.16 \\ \hline 
31 & \ 220.368  \\ \hline 
35 & \ 60.478 \\ \hline 
40 & \ 240.740 \\ \hline 
53 & \ 300.467  \\ \hline 
54 & \ 358.488 \\ \hline 
56 & \ 569.003  \\ \hline 
58 & \ 84.0687  \\ \hline 
60 & \ 464.238 \\ \hline 
66 & \ 188.872  \\ \hline 
67 & \ 411.408 \\ \hline 
74 & \ 352.617  \\ \hline 
83 & \ 106.822  \\ \hline 
96 & \ 135.124  \\ \hline 
97 & \ 299.172  \\ \hline 
99 & \ 318.033  \\ \hline 
101 & \ 210.934  \\ \hline 
103 & \ 676.431\\ \hline 
109 & \ 445.940 \\ \hline 
110 & \ 133.839 \\ \hline 
117 & \ 389.821 \\ \hline 
126 & \ 224.244  \\ \hline 
127 & \ 325.751 \\ \hline 
130 & \ 9.8125  \\ \hline 
133 & \ 4.7093  \\ \hline 
139 & \ 162.889 \\ \hline 
142 & \ 216.787 \\ \hline 
145 & \ 34.7298 \\ \hline 
146 & \ 1808.222 \\ \hline 
153 & \ 309.851  \\ \hline 
154 & \ 2.6791  \\ \hline
160 & \ 1137.254 
  \end{tabular}
 \end{minipage}%
 \begin{minipage}{0.5\linewidth}
  \caption{Periodic Algorithm \eqref{Alg: MR} expansions within $1000$ steps in $\mathbb{Q}_{43}$.}
  \bigskip
  \centering
  \begin{tabular}{c|c}
   D & $(A_{1000}/B_{1000})^2$ \\
   \hline
6 & \ 5.9995  \\ \hline 
14 & \ 13.9995  \\ \hline 
17 & \ 17.0065  \\ \hline 
21 & \ 20.9993  \\ \hline 
24 & \ 23.9992  \\ \hline 
38 & \ 37.9998  \\ \hline 
41 & \ 40.9996 \\ \hline 
44 & \ 43.9993  \\ \hline 
47 & \ 46.9992 \\ \hline 
52 & \ 51.9999  \\ \hline 
57 & \ 56.9994 \\ \hline 
59 & \ 58.9992 \\ \hline 
68 & \ 67.9998 \\ \hline 
78 & \ 77.9989\\ \hline 
79 & \ 78.9983 \\ \hline 
84 & \ 83.9990 \\ \hline 
87 & \ 86.9985 \\ \hline 
90 & \ 89.9993 \\ \hline 
92 & \ 91.9987\\ \hline 
95 & \ 94.9981\\ \hline 
102 & \ 101.999 \\ \hline 
107 & \ 106.998 \\ \hline 
111 & \ 110.998  \\ \hline 
122 & \ 121.998 \\ \hline 
124 & \ 123.999 \\ \hline 
135 & \ 134.998\\ \hline 
138 & \ 137.999  \\ \hline 
140 & \ 139.998  \\ \hline 
143 & \ 142.998  \\ \hline 
150 & \ 149.998 \\ \hline 
152 & \ 151.999 
  \end{tabular}
 \end{minipage}%
\end{table}

\newpage

\section*{Acknowledgments}
The author is a member of GNSAGA of INdAM.\\

\end{document}